\documentclass[11pt]{article}
\usepackage[a4paper,centering,hscale=0.7,vscale=0.75]{geometry}
\usepackage{amsfonts,amsmath,amsthm,amssymb}
\usepackage{mathtools}

\newtheorem{prop}{Proposition}[section]
\newtheorem{thm}[prop]{Theorem}
\newtheorem{lemma}[prop]{Lemma}
\newtheorem{coroll}[prop]{Corollary}

\theoremstyle{remark}
\newtheorem{rmk}[prop]{Remark}
\theoremstyle{definition}

\numberwithin{equation}{section}

\renewcommand{\P}{\mathbb{P}}
\newcommand{\E}{\mathbb{E}}
\newcommand{\F}{\mathbb{F}}
\renewcommand{\H}{\mathbf{H}}
\renewcommand{\L}{\mathbb{L}}
\newcommand{\LL}{\mathcal{L}}
\newcommand{\erre}{\mathbb{R}}
\newcommand{\enne}{\mathbb{N}}
\renewcommand{\epsilon}{\varepsilon}
\newcommand{\ds}{\displaystyle}

\DeclarePairedDelimiter\norm{\lVert}{\rVert}
\DeclarePairedDelimiterX\ip[2]{\langle}{\rangle}{#1,#2}
\title{On the maximal inequalities of Burkholder, Davis and Gundy}

\author{Carlo Marinelli\footnote{Department of Mathematics, University College London, Gower Street, London WC1E 6BT, UK.} \and Michael R\"ockner\footnote{Fakult\"at f\"ur Mathematik, Universit\"at Bielefeld, Germany.}}

\date{\normalsize 26 December 2012}
\begin{document}
\maketitle

\begin{abstract}
  We give a proof of the maximal inequalities of Burkholder, Davis and
  Gundy for real as well as Hilbert-space-valued local martingales
  using almost only stochastic calculus. Some parts of the exposition,
  especially in the infinite dimensional case, appear to be original.
\medskip\par\noindent
\emph{Keywords and phrases:} maximal inequalities, infinite
dimensional stochastic analysis, semimartingales.
\smallskip\par\noindent
\emph{2010 Mathematics Subject Classification:} 60G44
\end{abstract}

\section{Introduction}
The aim of this work is to provide a self-contained proof of the
Burkholder-Davis-Gundy (BDG) inequality for local c\`adl\`ag martingales,
both in finite and infinite dimension, using only stochastic calculus
and functional-analytic arguments. In particular, in the case of
real local martingales we provide a proof entirely based on stochastic
calculus for semimartingales, and, in the case of Hilbert-space-valued
local martingales, the proof uses some (relatively elementary)
techniques from duality of Banach spaces, interpolation of operators,
and vector measures. We also include a (known) proof for continuous
local martingales, entirely based on stochastic calculus, which serves
as motivation for the general case.
Even though we do not claim to have any original result, some of the
proofs appear to be new.

Let $(\Omega,\mathcal{F},\F,\P)$ be a filtered probability space and
$\H$ a real separable Hilbert space with norm $\norm{\cdot}$ and inner
product $(\cdot,\cdot)$. The goal is to prove the following theorem.
\begin{thm}     \label{thm:BDG}
  Let $M$ be an $\H$-valued $\F$-local c\`adl\`ag martingale with
  $M_0=0$. Then one has, for any $p \in [{1,\infty}[$ and for any
  $\F$-stopping time $\tau$,
  \begin{equation}
    \label{eq:BDG}
    \E [M,M]_\tau^{p/2} \lesssim_p \E \sup_{t \leq \tau} \|M_t\|^p
    \lesssim_p \E [M,M]_\tau^{p/2}.
  \end{equation}
  Moreover, if $M$ is continuous, then \eqref{eq:BDG} also holds for
  $p \in ]{0,2}[$, i.e.
  \begin{equation}
    \label{eq:BDGc}
    \E \sup_{t \leq \tau} \norm{M_t}^p \eqsim_p \E[M,M]_\tau^{p/2}
    = \E\ip{M}{M}_\tau^{p/2}
  \end{equation}
  for all $p \in ]{0,\infty}[$.
\end{thm}
\noindent Here $[M,M]$ and $\ip{M}{M}$ denote the (scalar) quadratic
variation and Meyer process of $M$, respectively. Throughout the paper
we use standard notation and terminology from the general theory of
processes, our main reference being \cite{Met}.
Setting $\L_p:=L_p(\Omega,\P)$ for convenience of notation, one can
equivalently write \eqref{eq:BDG} as
\[
\bigl\| M^*_\tau \bigr\|_{\L_p} \eqsim_p \bigl\| [M,M]_\tau^{1/2}
\bigr\|_{\L_p},
\]
where $M^*_t := \sup_{s \leq t} \norm{M_s}$.

A proof of the Davis inequality (i.e. of \eqref{eq:BDG} with $p=1$) for
real martingales that uses only stochastic calculus was given by
Meyer \cite{Mey:dual}, adapting the corresponding proof for real
continuous martingales by Getoor and Sharpe \cite{GetSha:conf}, the
key tool being a continuous-time version of Davis' decomposition,
also proved in \cite{Mey:dual}. The contribution of this paper is to
extend the stochastic calculus method of Meyer to the full range $p
\in [{1,\infty}[$ and to the Hilbert-space-valued case. As
we are going to see, a proof based only on stochastic calculus is
available for real martingales, but (unfortunately, perhaps) we have
not been able to provide such a proof for Hilbert-valued
martingales. Additional functional-analytic tools seem necessary. In
particular, we could not extend an $\L_p$ ($p>1$) inequality for
compensators of processes with integrable variation to the
Hilbert-valued case, although this problem was circumvented using a
different estimate (which involves the quadratic rather than the first
variation of compensators) and a simple duality argument.

While we are not aware of any standard reference where the BDG
inequalities for general local martingales (i.e. not necessarily
pathwise continuous) are proved by means of stochastic calculus only,
a proof based on Garsia-Neveu-type lemmata can be found, for instance,
in \cite{DM-mg,kall,LeLePr}. The latter proof covers also the more
general case (not considered here) where the $\L_p$ norm is replaced
by an Orlicz norm.

\medskip

We conclude this section with a few words about notation: we shall
write $a \lesssim b$ to mean that there exists a positive constant $N$
such that $a \leq Nb$. If the contant $N$ depends on the parameters
$p_1,\ldots,p_n$, we shall also write $N=N(p_1,\ldots,p_n)$ and
$\lesssim_{p_1,\ldots,p_n}$. The expression $a \eqsim b$ is equivalent
to $a \lesssim b \lesssim a$. 
The operator norm of an operator $T: E \to F$, with $E$ and $F$ two
Banach spaces, will be denoted by $\norm{T}_{E \to F}$. The spaces of
bounded linear and bilinear maps from $E$ to $F$ will be denoted by
$\mathcal{L}(E,F)$ and $\mathcal{L}_2(E,F)$, respectively.
Furthermore, every (local) martingale appearing below is assumed to be
c\`adl\`ag.

%-----------------------------------------------------------------------

\section{First steps} \label{sec:1st} Let us start with a simple but
helpful remark: given a local $\H$-valued martingale with respect to
the filtration $\F$, one has, for any $\F$-stopping time $\tau$,
$[M^\tau,M^\tau]_\infty=[M,M]_\tau$ and $M^*_\tau =
(M^\tau)^*_\infty$, where $M^\tau$ denotes the stopped martingale $t
\mapsto M_{t\wedge\tau}$. Therefore, by the monotone convergence
theorem, proving \eqref{eq:BDG} is equivalent to proving the
inequalities
\begin{equation}     \label{eq:bidigi}
  \norm[\big]{M^*_\infty}_{\L_p} \lesssim_p 
  \norm[\big]{[M,M]_\infty^{1/2}}_{\L_p} \lesssim_p 
  \norm[\big]{M^*_\infty}_{\L_p}.
\end{equation}
for any $\H$-valued martingale $M$ such that $M_0=0$. (We shall assume
throughout, without explicit mention, that all (local) martingale
start at zero). For convenience, we shall call ``upper bound'' and
``lower bound'' the left-hand and the right-hand inequality in
\eqref{eq:bidigi}, respectively.

In this section we are going to show, using only stochastic calculus,
that the upper bound holds for any $p \geq 2$, and that the lower
bound holds in $\L_{2p}$ provided the upper bound holds in
$\L_p$.
\begin{prop}[Upper bound, $p>2$]
  Let $M$ be an $\H$-valued martingale. One has, for any $p \in
  \left]2,\infty\right[$,
  \begin{equation}     \label{eq:ub}
  \bigl\| M^*_\infty \bigr\|_{\L_p} \lesssim_p 
  \bigl\| [M,M]_\infty^{1/2} \bigr\|_{\L_p}.
  \end{equation}
\end{prop}
\begin{proof}
  Let us introduce a sequence of stopping times $(T_n)_{n\in\enne}$
  defined by
  \[
  T_n := \inf\big\{t \geq 0:\; \|M_t\| > n \big\}.
  \]
  Let $n \in \enne$ be fixed for the time being, and consider the
  stopped martingale $M^{T_n}$.

  The first and second Fr\'echet derivatives of the function $\phi: \H
  \ni x \mapsto \norm{x}^p \in \erre$ at a generic point $x \in \H$ are
  given, respectively, by
  \[
  \H \simeq \LL(\H,\erre) \ni D\phi(x): u \mapsto p\norm{x}^{p-2}(x,u) 
  \]
  and
  \[
  \H \otimes \H \simeq \LL_2(\H,\erre) \ni D^2\phi(x): (u,v) \mapsto 
  p(p-1) \norm{x}^{p-4} (x,u) \, (x,v).
  \]
  It\^o's formula (in the form stated in \cite[Thm.~27.2,
  p.~190]{Met}) yields, denoting the bilinear form
  $(x,\cdot)\,(x,\cdot)$ by $x \otimes x$, and writing $T$
  instead of $T_n$ for simplicity,
  \begin{equation}     \label{eq:ito}
    \begin{split}
    \|M_T\|^p &= p\int_0^T \|M_{s-}\|^{p-2} M_{s-}\,dM_s\\
    &\quad + \frac12 p(p-1) \int_0^T \|M_s\|^{p-4} (M_s \otimes M_s)
            \,d[\![M,M]\!]^c_s\\
    &\quad + \sum_{s\leq T} \Bigl( \|M_s\|^p - \|M_{s-}\|^p 
             - p\|M_{s-}\|^{p-2} (M_{s-},\Delta M_s) \Bigr).
    \end{split}
  \end{equation}
  Let $(e_i)_{i\in\enne}$ be an orthonormal basis of $\H$, and set
  $M^i=\ip{M}{e_i}$ for all $i \in \enne$. Then we have, thanks to
  Kunita-Watanabe's inequality,
  \begin{align*}
    &\int_0^T \|M_s\|^{p-4} (M_s \otimes M_s) \,d[\![M,M]\!]^c_s\\
    &\qquad = \sum_{i,j} \int_0^T \|M_s\|^{p-4} M_s^iM_s^j d[M^i,M^j]_s^c\\
    &\qquad \leq \sum_{i,j} \left( \int_0^T \|M_s\|^{p-4}
      (M_s^i)^2\,d[M^j,M^j]_s^c \right)^{1/2}
    \left( \int_0^T \|M_s\|^{p-4} (M_s^j)^2\,d[M^i,M^i]_s^c \right)^{1/2}\\
    &\qquad =: \sum_{i,j} a_{ij}^{1/2} b_{ij}^{1/2} = \bigl\|
    a^{1/2}b^{1/2} \bigr\|_{\ell_1} \leq
    \bigl\| a^{1/2} \bigr\|_{\ell_2} \bigl\| b^{1/2} \bigr\|_{\ell_2}\\
    &\qquad = \left( \sum_{i,j} \int_0^T \|M_s\|^{p-4}
      (M_s^i)^2\,d[M^j,M^j]_s^c \right)^{1/2} \left( \sum_{i,j}
      \int_0^T \|M_s\|^{p-4} (M_s^j)^2\,d[M^i,M^i]_s^c \right)^{1/2}\\
    &\qquad = \int_0^T \|M_s\|^{p-2}d[M,M]_s^c
    \leq (M_T^*)^{p-2} \, [M,M]_T^c.
  \end{align*}
  Using Taylor's formula with remainder in Lagrange form, one
  obtains
  \begin{align*}
    &\sum_{s \leq T} \big( \norm{M_s}^p - \norm{M_{s-}}^p
    - p\|M_{s-}\|^{p-2} (M_{s-},\Delta M_s) \big) \\
    &\qquad \leq \frac{p(p-1)}{2} \sup_{s\leq T} \|M_{s-}\|^{p-2}
    \sum_{s \leq T} \|\Delta M_s\|^2 \leq \frac{p(p-1)}{2} 
    (M^*_T)^{p-2} \, [M,M]^d_T,
  \end{align*}
  hence also
  \begin{align*}
    &\frac12 p(p-1) \int_0^T \|M_s\|^{p-4} \bigl( M_s \otimes M_s \bigr)
    \,d[\![M,M]\!]^c_s\\
    &\qquad + \sum_{s \leq T} \Bigl( \|M_s\|^p -
    \|M_{s-}\|^p - p\|M_{s-}\|^{p-2} \ip{M_{s-}}{\Delta M_s} \Bigr)\\
    &\leq \frac{p(p-1)}{2} (M^*_T)^{p-2}
          \, \bigl( [M,M]^c_T + [M,M]^d_T \bigr)
    = \frac{p(p-1)}{2} (M^*_T)^{p-2}\, [M,M]_T.
  \end{align*}
  Since $M^*_T \leq n + [M,M]_T^{1/2}$ and $[M,M]_T^{1/2} \in \L_p$
  imply $M^*_T \in \L_p$, the first term on the right-hand side of
  \eqref{eq:ito} is a martingale, hence its expectation is zero.
  Therefore, taking expectation in \eqref{eq:ito}, Doob's and
  H\"older's inequalities yield
  \[
  \E(M_T^*)^p \lesssim_p \E\|M_T\|^p 
  \lesssim_p \bigl( \E(M_T^*)^p \bigr)^{\frac{p-2}{p}}
    \bigl( \E[M,M]_T^{p/2} \bigr)^{2/p},
  \]
  which is equivalent to
  \[
  \norm[\big]{M_{T_n}^*}_{\L_p} \lesssim_p
  \norm[\big]{[M,M]^{1/2}_{T_n}}_{\L_p}.
  \]
  Passing to the limit as $n \to \infty$, the proof is completed
  thanks to the monotone convergence theorem.
\end{proof}

\begin{coroll}[Lower bounds by upper bounds]
  Let $M$ be an $\H$-valued martingale, and assume that the upper bound
  \eqref{eq:ub} holds for some $p \geq 1$. Then one has
  \[
  \norm[\big]{[M,M]_\infty^{1/2}}_{\L_{2p}} \lesssim_p
  \norm[\big]{M^*_\infty}_{\L_{2p}}.
  \]
\end{coroll}
\begin{proof}
  The integration by parts formula $[M,M] = \norm{M}^2 - 2M_- \cdot M$
  yields
  \[ 
  [M,M]^{1/2}_\infty \leq M^*_\infty + \sqrt{2} \, 
  \sqrt{(M_- \cdot M)^*_\infty},
  \]
  which implies, by the previous proposition,
  \begin{align*}
    \bigl\| [M,M]^{1/2}_\infty \bigr\|_{\L_{2p}} &\leq
    \bigl\| M^*_\infty \bigr\|_{\L_{2p}} + \sqrt{2} \,
    \bigl\| (M_- \cdot M)^*_\infty \bigr\|_{\L_p}^{1/2}\\
    &\lesssim_p \bigl\| M^*_\infty \bigr\|_{\L_{2p}}
    + \bigl\| [M_-\cdot M,M_- \cdot M]_\infty^{1/2} \bigr\|_{\L_p}^{1/2}.
  \end{align*}
  One easily sees that, for any predictable process $\Phi$ with values
  in $\LL(\H,\erre) \simeq \H$, one has $[\Phi \cdot M,\Phi \cdot M]
  \leq \norm{\Phi} \cdot [M,M]$. Therefore, by Young's inequality, for
  any $\varepsilon>0$ there exists $N=N(\varepsilon)$ such that
  \begin{align*}
  [M_-\cdot M,M_- \cdot M]^{1/2}_\infty &\leq 
  \bigl( \|M_-\|^2 \cdot [M,M] \bigr)^{1/2}_\infty\\
  &\leq M_\infty^* \, [M,M]^{1/2}_\infty \leq 
  N (M_\infty^*)^2 + \varepsilon [M,M]_\infty,
  \end{align*}
  whence
  \begin{align*}
  \bigl\| [M_-\cdot M,M_- \cdot M]_\infty^{1/2} \bigr\|_{\L_p} 
  &\leq
  N \bigl\| (M_\infty^*)^2 \bigr\|_{\L_p}
  + \varepsilon \bigl\| [M,M]_\infty \bigr\|_{\L_p}\\
  &= N \bigl\| M_\infty^* \bigr\|^2_{\L_{2p}}
  + \varepsilon \bigl\| [M,M]^{1/2}_\infty \bigr\|^2_{\L_{2p}},
  \end{align*}
  which in turn implies
  \[
  \bigl\| [M,M]^{1/2}_\infty \bigr\|_{\L_{2p}} \lesssim_p
  (1+\sqrt{N}) \bigl\| M_\infty^* \bigr\|_{\L_{2p}} 
  + \sqrt{\varepsilon} \bigl\| [M,M]^{1/2}_\infty \bigr\|_{\L_{2p}}.
  \]
  The proof is completed by choosing $\varepsilon$ sufficiently small.
\end{proof}
\begin{rmk}
  We learned about the simple arguments of the above proofs in
  \cite{Met}, where the upper bound is stated as exercise~6.E.3. Later
  we found that the argument used in the proof of the lower bound is the
  same used by Getoor and Sharpe \cite{GetSha:conf} (in the simpler
  case of real continuous martingale), who in turn attribute it to
  Garsia (probably in the form of a personal communication or an
  unpublished manuscript). Metivier actually writes that the lower
  bound is ``easy'' for any $p \geq 2$.  Unfortunately we have not
  been able to find an easy proof for the case $2 \leq p < 4$. Let us
  also mention that the proof of the above corollary appears also in,
  e.g., \cite{RevYor} (for continuous real martingales, but the
  argument above is almost literally the same).
\end{rmk}

\section{Continuous local martingales}
We provide a proof of \eqref{eq:BDGc} based only on stochastic
calculus, following \cite{GetSha:conf} (this proof has been reproduced
\emph{verbatim} in some textbooks, see e.g. \cite{IkWa,Shgkw}).  This
way one can clearly see the arguments that will be used in Section
\ref{sec:gen} to treat discontinuous martingales, without the many
complications that appear in the general case.

The main idea is that, since we do not have ``tools'' to estimate the
$\L_p$-norms of $M^*_\infty$ and of $[M,M]^{1/2}_\infty$, we try to
reduce to a situation where estimates of $\L_2$-norms of the maximum
and of the quadratic variation of an auxiliary local martingale $N$
(that are already known to hold) would suffice. Reading the proofs in
\cite{GetSha:conf}, it might not be clear, at least at a first sight,
how the auxiliary martingales $N$ are chosen. Our (very minor)
contribution is to show why it is natural to choose precisely those
$N$.

The proof of \eqref{eq:BDGc} is split in several propositions.

\begin{prop}[Upper bound, $p<2$]     \label{prop:ub2-c}
  Let $M$ be an $\H$-valued continuous martingale. One has,
  for any $p \in ]0,2]$,
  \[
  \bigl\| M^*_\infty \bigr\|_{\L_p} \lesssim_p
  \bigl\| [M,M]^{1/2}_\infty \bigr\|_{\L_p} =
  \bigl\| \ip{M}{M}^{1/2}_\infty \bigr\|_{\L_p}. 
  \]
\end{prop}
\begin{proof}
  We apply the ``principle'' outlined above, that is, we look for an
  auxiliary (local) martingale $N$ which is more manageable than $M$,
  and we exploit the inequality $\E(N_\infty^*)^2 \lesssim
  \E[N,N]_\infty$. Let us set $N=H \cdot M$, where the integrand $H$
  is a predictable real process to be determined. Note that the
  identity $[N,N]=H^2\cdot [M,M]$ holds and, by the fundamental
  theorem of calculus,
  \[ 
  [M,M]_\infty^{p/2} = \frac{p}{2} \int_0^\infty
  [M,M]_s^{p/2-1}\,d[M,M]_s.
  \]
  It is thus natural to choose
  $H=\sqrt{p/2}\bigl(\varepsilon + [M,M]\bigr)^{p/4-1/2}$, so that
  $[N,N]=(\varepsilon+[M,M])^{p/2}$, where $\varepsilon>0$ is
  introduced to avoid singularities. Let us now try to obtain a
  relation between $N^*_\infty$ and $M^*_\infty$. Observe that, by the
  associativity property of the stochastic integral, we have $M=H^{-1}
  \cdot N$, as well as, by the integration by parts formula,
  \[
  M_\infty = (H^{-1} \cdot N)_\infty = H_\infty^{-1} N_\infty 
  - \int_0^\infty N_s \,dH^{-1}_s,
  \]
  which implies (taking into account that $s \mapsto H_s$ is
  decreasing, hence $s \mapsto H_s^{-1}$ is increasing)
  \[
  \|M_\infty\| \leq H^{-1}_\infty N^*_\infty 
  + N^*_\infty \int_0^\infty dH^{-1}_s = 2H^{-1}_\infty N^*_\infty,
  \]
  thus also $M^*_\infty \leq 2H^{-1}_\infty N^*_\infty$, as well as
  $\|M^*_\infty\|_{\L_p} \leq 2\|H^{-1}_\infty N^*_\infty\|_{\L_p}$.
  In order to obtain an expression involving the $\L_2$ norm of
  $N_\infty^*$, we apply H\"older's inequality in the form
  \[
  \| XY \|_{\L_p} \leq \norm{X}_{\L_2} \, \|Y\|_{\L_q}, \qquad
  \frac1p = \frac12 + \frac1q,
  \]
  which yields
  \[
  \bigl\| M^*_\infty \bigr\|_{\L_p} 
  \leq 2 \bigl\| H^{-1}_\infty N^*_\infty \bigr\|_{\L_p}
  \leq 2 \bigl\| H^{-1}_\infty \bigr\|_{\L_q} 
  \bigl\| N^*_\infty \bigr\|_{\L_2}.
  \]
  By the definition of $H$ and the identity $q=2p/(2-p)$, one has
  \[
  \norm[\big]{H^{-1}_\infty}_{\L_q} 
  = \sqrt{2/p} \norm[\big]{(\varepsilon+[M,M])^{1/2}_\infty}^{1-p/2}_{\L_p},
  \]
  as well as
  \[
  \norm[\big]{N^*_\infty}_{\L_2} 
  \leq 2 \norm[\big]{[N,N]^{1/2}_\infty}_{\L_2}
  = 2 \bigl\| (\varepsilon+[M,M])^{1/2}_\infty \bigr\|_{\L_p}^{p/2},
  \]
  which allows us to conclude that
  \[
  \bigl\| M^*_\infty \bigr\|_{\L_p} \leq 4 \sqrt{2/p} \,
  \bigl\| (\varepsilon+[M,M])^{1/2}_\infty \bigr\|_{\L_p}.
  \]
  The proof is finished by observing that $\varepsilon>0$ is arbitrary,
  hence the previous inequality also holds with $\varepsilon=0$.
\end{proof}

\begin{prop}[Lower bound, $p>2$]
  Let $M$ be an $\H$-valued continuous martingale. One has, 
  for any $p \in ]2,\infty]$,
  \[
  \bigl\| \ip{M}{M}_\infty^{1/2} \bigr\|_{\L_p} =
  \bigl\| [M,M]_\infty^{1/2} \bigr\|_{\L_p} \lesssim_p 
  \bigl\| M^*_\infty \bigr\|_{\L_p}.
  \]
\end{prop}
\begin{proof}
  Let us set, in analogy to the proof of the previous proposition,
  \[
  N := H \cdot M, \qquad H := \sqrt{p/2}[M,M]^{p/4-1/2},
  \]
  so that $[N,N]=[M,M]^{p/2}$, which implies
  \[
  \bigl\| [M,M]_\infty^{1/2} \bigr\|_{\L_p}^{p/2} 
  = \bigl\| [N,N]_\infty^{1/2} \bigr\|_{\L_2}
  = \bigl\| N_\infty \bigr\|_{\L_2(H)}
  \leq \bigl\| N^*_\infty \bigr\|_{\L_2}.
  \]
  The integration-by-parts formula yields
  \[
  N_\infty = (H \cdot M)_\infty = H_\infty M_\infty - \int_0^\infty M_s\,dH_s,
  \]
  from which one infers, since $s \mapsto H_s$ is increasing (because
  $p/2-1>0$), that $\|N_\infty\| \leq N_\infty^* \leq 2 H_\infty
  M^*_\infty$. This in turn implies
  \[
  \bigl\| N^*_\infty \bigr\|_{\L_2} 
  \leq 2 \bigl\| H_\infty M^*_\infty \bigr\|_{\L_2}
  \leq 2 \bigl\| H_\infty \bigr\|_{\L_q} \bigl\| M^*_\infty \bigr\|_{\L_p},
  \]
  where $1/2=p^{-1}+q^{-1}$, i.e. $q=2p/(p-2)$. Using the definition
  of $H$, one has
  \[
  \bigl\| H_\infty \bigr\|_{\L_q} 
  = \sqrt{p/2} \, \bigl\| [M,M]_\infty^{1/2} \bigr\|_{\L_p}^{p/2-1},
  \]
  hence
  \[
  \bigl\| [M,M]_\infty^{1/2} \bigr\|_{\L_p}^{p/2} 
  \leq \bigl\| N^*_\infty \bigr\|_{\L_2}
  \leq \sqrt{p/2} \, \bigl\| [M,M]_\infty^{1/2} \bigr\|_{\L_p}^{p/2-1}
  \bigl\| M^*_\infty \bigr\|_{\L_p},
  \]
  which implies the results by simplifying and rearranging terms.
\end{proof}

Note that both inequalities proved in the last two propositions relied
on (essentially) the same auxiliary local martingale. However,
unfortunately it seems difficult to use once again the same
construction to prove the lower bound in the case $p \in ]0,2[$. One
can immediately convince himself about this by inspection of the proof
of Proposition \ref{prop:ub2-c}. On the other hand, a similar proof
will still do, provided a different auxiliary martingale is used.
\begin{prop}[Lower bound, $p<2$]     \label{prop:lb2-c}
   Let $M$ be an $\H$-valued continuous martingale. One has, for any $p
  \in ]0,2[$,
  \[
  \bigl\| [M,M]_\infty^{1/2} \bigr\|_{\L_p} \lesssim_p 
  \bigl\| M^*_\infty \bigr\|_{\L_p}.
  \]
\end{prop}
\begin{proof}
  We introduce once again an auxiliary local martingale $N:= H \cdot
  M$, and then we compare the $\L_2$-norms of $[N,N]_\infty^{1/2}$ and
  $N^*_\infty$. It is (intuitively) clear that, in order to exploit
  the inequality $\norm[\big]{[N,N]_\infty^{1/2}}_{\L_2} \lesssim
  \norm[\big]{N^*_\infty}_{\L_2}$, one would need to establish an
  upper bound for $N^*_\infty$ in terms of $M^*_\infty$. For this
  purpose, let us use once again the integration-by-parts formula,
  assuming that $H$ is a real predictable process with finite variation
  which will be defined later. Then one has
  \[
  N_t = (H \cdot M)_t = H_t M_t 
  + \int_0^t M_s\,d(-H_s) \qquad \forall t>0.
  \]
  This ``starting point'' already suggests how to choose $H$: in fact,
  neglecting the integral on the right hand side, we see that
  $\E\|N_t\|^2$ should be of the order of $\E H_t^2(M_t^*)^2$, and we
  would like this term to be of the order of $\E (M_t^*)^p$, which
  suggests that we may try taking $H$ of the order of $(M^*)^{p/2-1}$.
  Let us then set
  \[
  H := \bigl( \varepsilon + M^* \bigr)^{p/2-1},
  \]
  where $\varepsilon>0$ is arbitrary and is introduced to avoid
  singularities. The identity $[N,N]=H^2 \cdot [M,M]$ implies
  $[N,N]_\infty \geq H^2_\infty [M,M]_\infty$, because $s \mapsto
  H^2_s$ is decreasing. Similarly, the integration-by-parts formula,
  the definition of $H$, and elementary calculus imply the estimate
  \begin{align*}
  N^*_\infty &\leq H_\infty M^*_\infty + \int_0^t M^*_s\,d(-H_s)\\
  &\leq \bigl( \varepsilon + M^*_\infty \bigr)^{p/2}
  + \int_0^\infty (\varepsilon + M^*_s) \,
  d\bigl(-(\varepsilon + M^*_s)^{p/2-1}\bigr)\\
  &\leq \bigl( \varepsilon + M^*_\infty \bigr)^{p/2}
  + ( 1-p/2 ) \int_0^\infty
  (\varepsilon + M^*_s)^{p/2-1}\,d(\varepsilon+M^*_s)\\
  &\leq \frac2p \bigl( \varepsilon + M_\infty^* \bigr)^{p/2}.
  \end{align*}
  Collecting estimates and taking $\L_2$-norms, we have
  \begin{align*}
  \bigl\| H_\infty [M,M]^{1/2}_\infty \bigr\|_{\L_2} 
  &\leq \bigl\| [N,N]^{1/2}_\infty \bigr\|_{\L_2}\\
  &\leq \bigl\| N^*_\infty \bigr\|_{\L_2}
  \leq \frac2p \bigl\| \bigl( 
       \varepsilon + M_\infty^* \bigr)^{p/2} \bigr\|_{\L_2}
  = \frac2p \bigl\| \varepsilon + M_\infty^* \bigr\|_{\L_p}^{p/2}.
  \end{align*}
  In order to obtain a term depending on the $\L_p$ norm of
  $[M,M]_\infty^{1/2}$ on the left-hand side, we proceed as follows:
  let $q>0$ be defined by the relation $p^{-1}=1/2+q^{-1}$,
  i.e. $q=2p/(2-p)$, and write, using H\"older's inequality,
  \begin{align*}
  \bigl\| [M,M]^{1/2}_\infty \bigr\|_{\L_p}
  &= \bigl\| H_\infty^{-1} H_\infty [M,M]^{1/2}_\infty \bigr\|_{\L_p}
  \leq \bigl\| H_\infty^{-1} \bigr\|_{\L_q}
  \bigl\| H_\infty [M,M]^{1/2}_\infty \bigr\|_{\L_2}\\
  &\leq \frac2p \, \bigl\| H_\infty^{-1} \bigr\|_{\L_q} 
  \bigl\| \varepsilon + M_\infty^* \bigr\|_{\L_p}^{p/2},
  \end{align*}
  where, by the definition of $H$ and elementary computations, $\bigl\|
  H_\infty^{-1} \bigr\|_{\L_q} = \bigl\| \varepsilon + M_\infty^*
  \bigr\|_{\L_p}^{1-p/2}$. We have thus proved the inequality $\bigl\|
  [M,M]^{1/2}_\infty \bigr\|_{\L_p} \leq (2/p)\,\bigl\| \varepsilon +
  M_\infty^* \bigr\|_{\L_p}$, which is valid also for $\varepsilon=0$,
  since $\varepsilon$ is arbitrary. The proof is thus finished.
\end{proof}

\begin{rmk}
  It is possible to give an alternative very short proof of
  \eqref{eq:BDGc} that involves little more than just It\^o's
  formula. In fact, appealing to Lenglart's domination inequality (see
  \cite{Lenglart}), one can show that once either the (lower or upper)
  bound holds in $\L_p$, then it holds in $\L_q$ for all $q \in
  ]0,p[$. In particular, the bounds of Section \ref{sec:1st} are
  enough to prove Theorem \ref{thm:BDG} for continuous local
  martingales (cf. \cite{RevYor} for more detail).  This method,
  however, does not work for general local martingales.
\end{rmk}

%--------------------------------------------------------------------

\section{Auxiliary results}
\subsection{Calculus for functions of finite variation}
We shall denote the variation of a function $f:\erre_+ \to H$ by
$\int_0^\infty |df|$. Recall that, if $f$ has finite variation and
$f(0)=0$, then $f^*$ is bounded by the variation of $f$ (in fact,
$f^*$ is bounded by the oscillation of $f$, which is in turn bounded
by the variation of $f$).

\bigskip

Let $U$ and $V$ be two $\H$-valued functions with finite
variation. Then the following integration-by-parts formula holds
\[
\ip{U}{V} = U_- \cdot V + V_- \cdot U + \sum (\Delta U,\Delta V).
\]
Since the series in the previous expression can be written as
$(V-V_-)\cdot U$, one also has
\begin{equation}     \label{eq:ibp-fv}
(U,V) = U_- \cdot V + V \cdot U \equiv \int U_-\,dV + \int V\,dU.
\end{equation}
Using the integration-by-parts formula for semimartingales, it is
immediately seen that \eqref{eq:ibp-fv} still holds if only one of
$U$ and $V$ is a process with finite variation and the other one is a
semimartingale (and appropriate measurability conditions are satisfied).

Calculus rules for functions (and processes) of finite variation may
differ substantially from the ``classical'' calculus rules for
continuous functions. In this section we collect some elementary
identities that will be needed in the sequel. In particular, if $U
\equiv V$, \eqref{eq:ibp-fv} yields
\begin{equation}    \label{eq:fv-r1}
  dU^2 = (U_- + U)\,dU,  
\end{equation}
which also implies, if $U>0$,
\begin{equation}     \label{eq:fv-r2}
  d(U^{1/2}) = \frac{1}{U_-^{1/2} + U_{\phantom{-}}^{1/2}}\,dU.  
\end{equation}
Assume now $U \geq \delta$ for some $\delta>0$, and set
$V=1/U$. Then \eqref{eq:ibp-fv} yields
\[
1 = U/U = -\int U_-\,d(-1/U) + \int (1/U)\,dU,
\]
or equivalently
\begin{equation}     \label{eq:fv-r3}
d(-1/U) = \frac{1}{UU_-}\,dU.  
\end{equation}

\begin{lemma}     \label{lm:fv1}
  Let $V$ be an increasing function with $V_0=0$. Then one has
  \[
  \int_0^t V_{s-}\,d(V_s^{q-1}) \leq \frac{q-1}{q} V_t^q \qquad
  \forall q \in ]{1,\infty}[
  \]
  and
  \[
  \int_0^t V_{s-} \,d(-V_s^{q-1}) \leq \frac{1-q}{q} V_t^q \qquad
  \forall q \in ]{0,1}[.
  \]
\end{lemma}
\begin{proof}
  Let $0=t_0 < t_1 < \cdots < t_n=t$ be a finite partition of
  $[0,t]$. For $q \geq 1$, one has
  \[
  \int_0^t V_{s-} \,d(V_s^{q-1}) =
  \lim_{n\to\infty} \sum_{i=0}^{n-1} V_{t_i} \bigl( V_{t_{i+1}}^{q-1} -
  V_{t_i}^{q-1} \bigr),
  \]
  hence the first inequality is proved if we can show that
  \[
  x_1 \bigl( x_2^{q-1}-x_1^{q-1} \bigr) \leq \frac{q-1}{q} \bigl( x_2^q-x_1^q
  \bigr) \qquad \forall \, 0 \leq x_1 \leq x_2.
  \]
  In fact, by the fundamental theorem of calculus, one has
  $x_2^q-x_1^q \leq q(x_2-x_1)x_2^{q-1}$, which implies, after a few
  elementary computations,
  \[
  x_2^q \geq -\frac{1}{q-1}x_1^q + \frac{q}{q-1}x_1x_2^{q-1},
  \]
  hence also
  \[
  \frac{q-1}{q} \bigl( x_2^q-x_1^q \bigr) \geq x_1 \bigl( 
  x_2^{q-1}-x_1^{q-1} \bigr).
  \]
  Let us now turn to the case $0<q<1$. Note that, in principle, we
  cannot write
  \[
  \int_0^t V_{s-} \,d(-V_s^{q-1}) =
  \lim_{n\to\infty} \sum_{i=0}^{n-1} V_{t_i} \bigl( V_{t_i}^{q-1} -
  V_{t_{i+1}}^{q-1} \bigr)
  \]
  because $V_0=0$. However, a simple regularization of the type
  $V_0=\varepsilon>0$ and then passing to the limit as $\varepsilon
  \to 0$ at the end of the computations would suffice. Hence we can
  proceed in a slightly formal (but harmless) way accepting the
  previous identity as true, and, in analogy to the case $q>1$, it is
  enough to show that
  \[
  V_{t_i} \bigl( V_{t_i}^{q-1} - V_{t_{i+1}}^{q-1} \bigr) \leq
  \frac{1-q}{q} \bigl( V_{t_{i+1}}^q - V_{t_i}^q \bigr)
  \qquad \forall i \in \{0,1,\ldots,n-1\}.
  \]
  The latter inequality certainly holds true if one has
  \[
  x_2^q - x_1^q \geq \frac{q}{1-q} x_1(x_1^{q-1}-x_2^{q-1}) 
  \qquad \forall 0 \leq x_1 \leq x_2.
  \]
  Let $0 \leq x_1 \leq x_2$. Since $x \mapsto x^{q-1}$ is decreasing,
  the fundamental theorem of calculus yields
  \[
  x_2^q - x_1^q = q\int_{x_1}^{x_2} y^{q-1}\,dy \geq q x_2^q - q x_1
  x_2^{q-1}.
  \]
  Rearranging terms, this implies
  \[
  x_2^q \geq \frac{1}{1-q}x_1^q - \frac{q}{1-q} x_1 x_2^{q-1},
  \]
  hence also
  \[
  x_2^q - x_1^q \geq \frac{q}{1-q} \, \bigl( x_1^q - x_1x_2^{q-1} \bigr),
  \]
  which is the desired inequality.
\end{proof}

We shall also need some $\L_p$ estimates for compensators of processes
with integrable variation.
\begin{prop}     \label{prop:BDG-fv}
  Let $V$ be a real increasing process with compensator
  $\tilde{V}$. Then
  \[
  \left\| \int_0^\infty |d\tilde{V}| \, \right\|_{\L_p} \leq p \left\|
    \int_0^\infty |dV| \, \right\|_{\L_p} \qquad \forall 1 \leq p <
  \infty.
  \]
\end{prop}
\begin{proof}
  Since $V$ is increasing, then $\tilde{V}$ is also
  increasing\footnote{An overkill proof of this fact is that $V$ is a
    submartingale, and $V=(V-\tilde{V})+\tilde{V}$ is its Doob-Meyer
    decomposition.}, hence we only have to prove
  $\norm[\big]{\tilde{V}_\infty}_{\L_p} \leq p
  \norm[\big]{V_\infty}_{\L_p}$.  We have
  \[
  \norm[\big]{\tilde{V}_\infty}_{\L_p} = \sup_{\xi \in B_1(\L_q)}
  \E\xi\tilde{V}_\infty,
  \]
  where $q$ is the conjugate exponent of $p$ and $B_1(\L_q)$ stands
  for the unit ball of $\L_q$. Let $\xi \in B_1(\L_q)$ be arbitrary
  but fixed, and introduce the martingale $N$ defined by
  $N_t:=\E[\xi|\mathcal{F}_t]$, $t \geq 0$, $N_\infty:=\xi$. The
  integration-by-parts formula yields
  \[
  \xi \tilde{V}_\infty = N_\infty \tilde{V}_\infty = (\tilde{V} \cdot N)_\infty
  + (N_- \cdot \tilde{V})_\infty,
  \]
  hence, using the definition of compensator, the fact that $V$ is
  increasing, and H\"older's inequality, one obtains
  \[
  \E \xi \tilde{V}_\infty = \E( N_- \cdot V)_\infty \leq 
  \E N^*_\infty V_\infty \leq 
  \norm[\big]{N^*_\infty}_{\L_q} \norm[\big]{V_\infty}_{\L_p}.
  \]
  Since, by Doob's inequality, one has
  \[
  \norm[\big]{N^*_\infty}_{\L_q} \leq p \, \norm[\big]{N_\infty}_{\L_q} =
  p \, \norm[\big]{\xi}_{\L_q} \leq p,
  \]
  the conclusion follows because $\xi$ is arbitrary.
\end{proof}

The following proposition extends, in the case $p=1$, the previous
inequality to Hilbert-space-valued processes.
\begin{prop}     \label{prop:veme}
  Let $X: \erre_+ \to H$ be a right-continuous measurable process such
  that $\E\int |dX|<\infty$. Then $X$ admits a dual predictable
  projection (compensator) $\tilde{X}$, which satisfies
  \begin{equation}     \label{eq:dincu}
  \E\int_0^\infty |d\tilde{X}| \leq \E\int_0^\infty |dX|.
  \end{equation}
\end{prop}
\begin{proof}
  The existence of $\tilde{X}$, as a process with integrable variation,
  follows by a result due to Dinculeanu (see \cite[Thm.~22.8,
  p.~278]{Dinc}).
  Let $\mathcal{M} := \mathcal{B}(\erre_+) \otimes \mathcal{F}$.  By
  \cite[Thm.~19.8, p.~220]{Dinc}, there exists a $\sigma$-additive
  measure $\mu_X: \mathcal{M} \to H$ with finite variation
  $|\mu_X|$ such that
  \[
  \mu_X(A) = \E\int_0^\infty 1_A \,dX \qquad \forall A \in \mathcal{M}
  \]
  and
  \[
  |\mu_X|(A) = \E\int_0^\infty 1_A \,|dX| \qquad \forall A \in \mathcal{M}.
  \]
  The dual predictable projection $\mu_X^p$ of $\mu_X$ is defined by
  \[
  \mu_X^p: \mathcal{M} \ni A \mapsto \int_{\erre_+ \times \Omega}
  \prescript{p}{}{1}_A\,d\mu_X,
  \]
  where $\prescript{p}{}{Y}$ denotes the predictable projection of a
  measurable process $Y$.
The dual predictable projection
  (compensator) $\tilde{X}$ is constructed as the unique process
  associated to the measure $\mu^p_X =: \mu_{\tilde{X}}$.  We can thus
  write
  \[
  \E\int_0^\infty |d\tilde{X}| 
  = \int_{\erre_+ \times \Omega} 1\,d|\mu_{\tilde{X}}|
  = \int_{\erre_+ \times \Omega} 1\,d|\mu^p_X|,
\]
where, by \cite[Thm.~22.1, p.~272]{Dinc}, $|\mu^p_X| \leq |\mu_X|^p$.
 Taking into account that, if there exists a
  constant $N$ such that $|Y| \leq N$ then $|\prescript{p}{}{Y}| \leq
  N$ outside an evanescent set (see e.g. \cite[\S{VI.43}]{DM-mg}), one
  has, by the definition of dual predictable projection of a measure,
\[
\E\int_0^\infty |d\tilde{X}| 
  \leq \int_{\erre_+ \times \Omega} 1\,d|\mu_X|^p
\leq \int_{\erre_+ \times \Omega} 1\,d|\mu_X| = \E\int_0^\infty |dX|.
\qedhere
\]
\end{proof}

\subsection{An extension of the Riesz-Thorin interpolation
theorem}     \label{ssec:RT}
We quote, omitting the proof, a generalization of the Riesz-Thorin
interpolation theorem, dealing with $L_p$ spaces with mixed norm.

Let $(X_1,\mu_1)$, $(X_2,\mu_2),\ldots,(X_n,\mu_n)$ be measure spaces,
and $p_1,p_2,\ldots,p_n \in [1,\infty]$. Setting
$\mathbf{p}=(p_1,p_2,\ldots,p_n)$ and
$1/\mathbf{p}:=(1/p_1,1/p_2,\ldots,1/p_n)$ for convenience of
notation, let us define the following spaces of integrable functions
with mixed norm:
\[
L_{\mathbf{p}} := L_{p_1,\ldots,p_n} 
:= L_{p_1}(X_1 \to L_{p_2,\ldots,p_n},\mu_1),
\quad \ldots \quad
,L_{p_n}:=L_{p_n}(X_n,\mu_n).
\]
The following result is due to Benedek and Panzone \cite[p.~316]{BenPan}.
\begin{thm}     \label{thm:interp}
  Let $T$ be a linear operator such that
  \[
  \norm[\big]{T}_{L_{\mathbf{p}_0} \to L_{\mathbf{q}_0}} \leq M_0, \qquad
  \norm[\big]{T}_{L_{\mathbf{p}_1} \to L_{\mathbf{q}_1}} \leq M_1,
  \]
  with $\mathbf{p}_0, \mathbf{p}_1, \mathbf{q}_0, \mathbf{q}_1 \in \,
  [1,\infty]^n$.
  Let $\theta \in \, ]{0,1}[$ and define $\mathbf{p}$, $\mathbf{q}$
  through
  \[
  \frac{\theta}{\mathbf{p}_1} + \frac{1-\theta}{\mathbf{p}_0} =
  \frac{1}{\mathbf{p}}, \qquad \frac{\theta}{\mathbf{q}_1} +
  \frac{1-\theta}{\mathbf{q}_0} = \frac{1}{\mathbf{q}}.
  \]
  Then one has
  \[
  \norm[\big]{T}_{L_{\mathbf{p}} \to L_{\mathbf{q}}} \leq M_0^{1-\theta} M_1^\theta.
  \]
\end{thm}

\medskip

Let $(X_1,\mu_1)$, $(X_2,\mu_2)$ be two measure spaces, and $H$ a real separable Hilbert space. 
We shall say that a map $T:L_0(X_1 \to H,\mu_1) \to L_0(X_2,\mu_2)$ is \emph{sublinear}
if
\begin{itemize}
\item[(a)] $|T(f_1+f_2)| \leq |Tf_1| + |Tf_2|$ for all $f_1$, $f_2 \in
  L_0(X_1 \to H,\mu_1)$;
\item[(b)] $|T(\alpha f)| = |\alpha| |Tf|$ for all $\alpha \in \erre$
  and $f \in L_0(X_1 \to H,\mu_1)$.
\end{itemize}
The following result can be deduced by the previous theorem identifying $H$ with $\ell_2$, and by using the linearization method of \cite{Jans:interp} (cf. also \cite{CZ:interp}) to cover the case of sublinear operators.
\begin{thm}     \label{thm:CZ}
  Let $T$ be a sublinear operator such that
  \[
  \norm[\big]{T}_{L_{p_0}(X_1 \to H) \to L_{q_0}(X_2)} \leq M_0, \qquad
  \norm[\big]{T}_{L_{p_1}(X_1 \to H) \to L_{q_1}(X_2)} \leq M_1,
  \]
  with $p_0$, $p_1$, $q_0$, $q_1 \in \, [1,\infty]$.
  Let $\theta \in \, ]{0,1}[$ and define $p$, $q$
  through
  \[
  \frac{\theta}{p_1} + \frac{1-\theta}{p_0} =
  \frac{1}{p}, \qquad \frac{\theta}{q_1} +
  \frac{1-\theta}{q_0} = \frac{1}{q}.
  \]
  Then one has
  \[
  \norm[\big]{T}_{L_p(X_1 \to H) \to L_q(X_2)} \leq M_0^{1-\theta} M_1^\theta.
  \]
\end{thm}

\subsection{Stein's estimate for predictable projections of
  discrete-time processes}
In the proof of the BDG inequality for Hilbert-space-valued
martingales we shall use a slightly extended version of an $\L_p$
estimate for the quadratic variation of the predictable projection of
an arbitrary discrete-time process, due to Stein (cf. \cite[Thm.~8,
p.~103]{Stein-LP}). For the sake of completeness, we include its
simple and elegant proof, which relies on Theorem \ref{thm:interp}
above.
\begin{thm}[Stein] \label{thm:stein} 
  Let $(\Omega,\mathcal{F},(\mathcal{F}_n)_{n\in\enne},\P)$ be a
  discrete-time stochastic basis and $(f_n)_{n\in\enne}$ an $\H$-valued
  process. For any $p \in ]{1,\infty}[$ and any sequence $(n_k)_{k \in
    \enne}$ of positive integers, denoting conditional expectation
  with respect to $\mathcal{F}_n$ by $\E_n$, one has
  \[
  \bigg\| \Bigl( \sum_k \norm[\big]{\E_{n_k} f_k}^2 \Bigr)^{1/2} 
  \bigg\|_{\L_p} \lesssim_p
  \bigg\| \Bigl( \sum_k \norm[\big]{f_k}^2 \Bigr)^{1/2} \bigg\|_{\L_p},
  \]
  i.e. $\norm[\big]{(\E_{n_k}f_k)_k}_{\L_p(\ell_2(\H))} \lesssim_p
  \norm[\big]{f}_{\L_p(\ell_2(\H))}$.
\end{thm}
\begin{proof}
  Define the linear operator
  \[
  T: (f_k) \mapsto \bigl(\E_{n_k} f_k\bigr).
  \]
  Let us show that $T$ is bounded on $\L_p(\ell_p(\H))$: one has, by
  the conditional Jensen inequality and Tonelli's theorem,
  \begin{align*}
    \norm[\big]{Tf}_{\L_p(\ell_p(\H))}^p &\equiv
    \norm[\big]{\bigl(\E_{n_k}f_k\bigr)}_{\L_p(\ell_p(\H))}^p =
    \E \sum_k \norm[\big]{\E_{n_k}f_k}^p\\
    &\leq \E \sum_k \E_{n_k}\|f_k\|^p =
    \sum_k \E\E_{n_k}\|f_k\|^p \leq
    \sum_k \E\|f_k\|^p\\
    &= \E \sum_k \|f_k\|^p = \norm[\big]{f}_{\L_p(\ell_p(\H))}^p.
  \end{align*}
  Let us also show that $T$ is bounded on $\L_p(\ell_\infty(\H))$: one
  has
  \begin{align*}
  \norm[\big]{Tf}_{\L_p(\ell_\infty(\H))}^p &\equiv
  \E\sup_k \norm[\big]{\E_{n_k}f_k}^p \leq 
  \E\sup_k \bigl(\E_{n_k}\|f_k\|\bigr)^p\\
  &\leq \E\sup_n \bigl(\E_n \sup_k \|f_k\| \bigr)^p =
  \E\sup_n \bigl| \xi^k_n \bigr|^p,
  \end{align*}
  where $n \mapsto \xi_n := \E_n \xi_\infty$, with $\xi_\infty:=\sup_k
  \|f_k\|$, is a real-valued martingale. Doob's maximal inequality
  then yields
  \[
  \E\sup_n \bigl| \xi_n \bigr|^p \lesssim_p \E\bigl| \xi_\infty \bigr|^p 
  = \E \bigl( \sup_k \|f_k\| \bigr)^p =
  \norm[\big]{f}_{\L_p(\ell_\infty(\H))}.
  \]
  Identifying $\H$ with $\ell_2$, we have shown that
  \[
  \norm[\big]{T}_{L_{p,p,2} \to L_{p,p,2}} < \infty, \qquad
  \norm[\big]{T}_{L_{p,\infty,2} \to L_{p,\infty,2}} < \infty,
  \]
  hence Theorem \ref{thm:interp} implies that
  \[
  \norm[\big]{T}_{\L_p(\ell_q(\H)) \to \L_p(\ell_q(\H))} =
  \norm[\big]{T}_{L_{p,q,2} \to L_{p,q,2}} < \infty 
  \qquad \forall 1 < p \leq q < \infty.
  \]
  In particular, this proves the theorem in the case $1 < p \leq
  2$. Let us show that $T$ is a bounded endomorphism of
  $\L_p(\ell_2(\H))$ also if $p>2$.
% To this purpose, we first show that
% $T$ is self-adjoint on $\L_2(\ell_2(\H))$. For this it suffices to
% show that one has
% \[
% \E \bigl(\E_n\phi)\psi = \E \phi\bigl(\E_n\psi) 
% \qquad \forall \phi, \, \psi \in \L_2(\ell_2(H)).
% \]
% In fact, one has
% \[
% \E \bigl(\E_n\phi)\psi = \E\E_n\bigl[ \bigl(\E_n\phi)\psi \bigr]
% = \E\bigl(\E_n\phi\bigr)\bigl(\E_n\psi\bigr),
% \]
% and, simply exchanging $\phi$ with $\psi$, $\E \phi\bigl(\E_n\psi) =
% \E\bigl(\E_n\phi\bigr)\bigl(\E_n\psi\bigr)$.  
  Let $p>2$ and $f \in \L_p(\ell_2(\H))$. Then one has, denoting the
  duality form between $\L_p(\ell_2(\H))$ and $\L_{p'}(\ell_2(\H))$ by
  $\ip{\cdot}{\cdot}$ and the unit ball of $\L_{p'}(\ell_2(\H))$ by
  $B_1$, taking into account that $T$ is self-adjoint on
  $\L_2(\ell_2(\H))$,
  \begin{align*}
    \norm[\big]{Tf}_{\L_p(\ell_2(\H))} &= \sup_{g\in B_1} \ip{Tf}{g}
    = \sup_{g \in B_1} \ip{f}{Tg}\\
    &\leq \sup_{g \in B_1} \norm{f}_{\L_p(\ell_2(\H))} \,
    \norm{Tg}_{\L_{p'}(\ell_2(\H))}\\
    &\leq \norm{f}_{\L_p(\ell_2(\H))} \, 
    \norm{T}_{\L_{p'}(\ell_2(\H)) \to \L_{p'}(\ell_2(\H))},
  \end{align*}
  where the operator norm of $T$ in the last term is finite because,
  as already proved, $T$ is a bounded endomorphism of
  $\L_{p'}(\ell_2(\H))$.
\end{proof}

%----------------------------------------------------------------------

\section{General case}     \label{sec:gen}
\subsection{Predictably bounded jumps}
\begin{prop}[Conditional lower bound, $p=1$]     \label{prop:clb1}
  Let $M$ be an $\H$-valued martingale and $D$ an increasing adapted
  process such that $\|\Delta M_-\| \leq D_-$. Then one has
  \[
  \norm[\big]{[M,M]_\infty^{1/2}}_{\L_1} \lesssim
  \norm[big]{M_\infty^* + D_\infty}_{\L_1}.
  \]
\end{prop}
\begin{proof}
  For a fixed $\varepsilon>0$, let us introduce the local martingale
  $N := H_- \cdot M$, where
  \[
  H := \bigl( \varepsilon + D + M^* \bigr)^{-1/2}.
  \]
  We are going to compare the $\L_2$ norms of $[N,N]_\infty$ and of
  $\|N_\infty\|^2$. To this purpose, note that, since $s \mapsto H_s$
  is decreasing, one has
  \[ 
  [N,N]_\infty = H_-^2 \cdot [M,M] \geq H_\infty^2 [M,M]_\infty,
  \]
  hence also $\norm[\big]{H_\infty [M,M]^{1/2}_\infty}_{\L_2} \leq
  \norm[\big]{[N,N]^{1/2}_\infty}_{\L_2}$. Moreover, the Cauchy-Schwarz
  inequality yields
  \begin{equation}    \label{eq:mdvf}
  \begin{split}
  \norm[\big]{[M,M]^{1/2}_\infty}_{\L_1}
  &= \norm[\big]{[M,M]^{1/2}_\infty H_\infty H_\infty^{-1}}_{\L_1}
  \leq \norm[\big]{[M,M]^{1/2}_\infty H_\infty}_{\L_2}
  \norm[\big]{H_\infty^{-1}}_{\L_2}\\
  &\leq \norm[\big]{[N,N]^{1/2}_\infty}_{\L_2} \norm[\big]{H_\infty^{-1}}_{\L_2}
  = \norm[\big]{N_\infty}_{\L_2(H)} \norm[\big]{H_\infty^{-1}}_{\L_2},
  \end{split}
  \end{equation}
  where
  \[
  \norm[\big]{H_\infty^{-1}}_{\L_2} =
  \norm[\big]{\bigl(\varepsilon + D_\infty + M_\infty^*\bigr)^{1/2}}_{\L_2}
  = \norm[\big]{\varepsilon + D_\infty + M_\infty^*}^{1/2}_{\L_1}
  \]
  Thanks to the integration-by-parts formula
  $H_- \cdot M=HM-M \cdot H$, one has
  \[
  \|N_\infty\| \leq H_\infty M^*_\infty + \int_0^\infty \|M_s\| \, dH_s,
  \]
  which implies, by the inequality $\|M_s\| \leq \|M_{s-}\| + \|\Delta
  M_s\| \leq \varepsilon + M^*_{s-} + D_{s-}$ and the definition of
  $H$,
  \[
  \|N_\infty\| \leq (\varepsilon + D_\infty + M^*_\infty)^{1/2} +
  \int_0^\infty (\varepsilon + M^*_{s-} + D_{s-})
  \,d\bigl(-(\varepsilon + D_s + M^*_s)^{-1/2}\bigr).
  \]  
  Setting $U:=H_-^{-1} \equiv (\varepsilon + D + M^*)^{1/2}$, the
  integral on the right hand side can be written as
  \[
  \int_0^\infty U_-^2\,d(-1/U) = \int_0^\infty \frac{U_-^2}{UU_-}\,dU
  \leq \int_0^\infty dU = (\varepsilon + D_\infty + M_\infty^*)^{1/2},
  \]
  hence
  \[
  \norm[\big]{N_\infty}_{\L_2(H)} \leq \norm[\big]{\bigl(\varepsilon + D_\infty
    + M_\infty^*\bigr)^{1/2}}_{\L_2} = \norm[\big]{\varepsilon + D_\infty +
    M_\infty^*}^{1/2}_{\L_1}.
  \]
  Taking \eqref{eq:mdvf} into account and recalling that
  $\varepsilon>0$ is arbitrary, the proof is completed.
\end{proof}

\begin{prop}[Conditional upper bound, $p=1$]     \label{prop:cub1}
  Let $M$ be an $\H$-valued martingale and $D$ an increasing adapted
  process such that $\|\Delta M_-\| \leq D_-$. Then one has
  \[
  \norm[\big]{M^*_\infty}_{\L_1} \lesssim 
  \norm[\big]{[M,M]^{1/2}_\infty + D_\infty}_{\L_1}.
  \]
\end{prop}
\begin{proof}
  We adapt the proof of Propositon \ref{prop:ub2-c}. Let
  \[
  H := \sqrt{1/2} \bigl( [M,M] + D^2 \bigr)^{-1/4}, 
  \qquad N:= H_- \cdot M,
  \]
  so that $M=H_-^{-1} \cdot M$ and, by the integration-by-parts
  formula,
  \[
  M_\infty = (H_-^{-1} \cdot N)_\infty = H_\infty^{-1}N_\infty - (N
  \cdot H^{-1})_\infty.
  \]
  Taking norm on both sides and recalling that $s \mapsto H_s^{-1}$ is
  increasing, one has
  \[
  M^*_\infty \leq H^{-1}_\infty N^*_\infty + N^*_\infty \int_0^\infty
  d(H_s^{-1}) = 2H^{-1}_\infty N^*_\infty,
  \]
  hence also, using the Cauchy-Schwarz and Doob inequalities,
  \begin{align*}
  \norm[\big]{M^*_\infty}_{\L_1} &\leq 
  2 \norm[\big]{H^{-1}_\infty N^*_\infty}_{\L_1}
  \leq 2 \norm[\big]{H^{-1}_\infty}_{\L_2} \norm[\big]{N^*_\infty}_{\L_2}\\
  &\leq 4 \norm[\big]{H^{-1}_\infty}_{\L_2} \norm[\big]{N_\infty}_{\L_2(H)}
  = 4 \norm[\big]{H^{-1}_\infty}_{\L_2} \norm[\big]{[N,N]^{1/2}_\infty}_{\L_2},
  \end{align*}
  where
  \[
  \norm[\big]{H^{-1}_\infty}_{\L_2} = \sqrt{2} \,
  \norm[\big]{\bigl(\varepsilon + [M,M]_\infty + D^2_\infty\bigl)^{1/4}}_{\L_2}
  = \sqrt{2} \,
  \norm[\big]{\bigl(\varepsilon + [M,M]_\infty + D^2_\infty\bigl)^{1/2}}_{\L_1}.
  \]
  Moreover, by definition of $N$, one has
  \[
  [N,N]_\infty = \frac12 \int_0^\infty 
  \bigl( [M,M]_{s-} + D_{s-}^2 \bigr)^{-1/2} \,d[M,M]_s,
  \]
  hence, noting that $[M,M]_{s-} + D_{s-}^2 \geq [M,M]_{s-} + \|\Delta M_s \|^2
  = [M,M]_s$ for all $s \geq 0$,
  \begin{align*}
  [N,N]_\infty &\leq \int_0^\infty \frac{1}{2[M,M]_s^{1/2}}\,d[M,M]_s\\
  &\leq \int_0^\infty \frac{1}{[M,M]_{s-}^{1/2} + [M,M]_s^{1/2}}\,d[M,M]_s
  = [M,M]_\infty^{1/2},
  \end{align*}
  where we have used \eqref{eq:fv-r2}, as well as the fact that
  $[M,M]$ is an increasing process and $x \mapsto x^{-1/2}$ is
  decreasing. This implies
  \[
  \norm[\big]{[N,N]_\infty^{1/2}}_{\L_2} = \norm[\big]{[M,M]_\infty^{1/4}}_{\L_2} = 
  \norm[\big]{[M,M]_\infty^{1/2}}^{1/2}_{\L_1} \leq
  \norm[\big]{\bigl(\varepsilon + [M,M]_\infty + D^2_\infty\bigl)^{1/2}}_{\L_1},
  \]
  and the proof is finished by combining the estimates.
\end{proof}

We now consider the case $p \in ]1,2[$, whose proof proceeds along the
lines of the previous one, but is technically more complicated.
\begin{prop}[Conditional lower bound, $1 < p < 2$]     \label{prop:clb2-}
  Let $M$ be an $\H$-valued martingale and $D$ an increasing adapted
  process such that $\|\Delta M_-\| \leq D_-$. Then one has, for any
  $p \in ]1,2[$,
  \[
  \bigl\| [M,M]_\infty^{1/2} \bigr\|_{\L_p} \lesssim_p \bigl\|
  M_\infty^* + D_\infty \bigr\|_{\L_p}.
  \]
\end{prop}
\begin{proof}
  Let us introduce the local martingale $N=H_- \cdot M$, with
  \[
  H := \bigl( \varepsilon + M^* + D \bigr)^{p/2-1}.
  \]
  Defining $q>0$ by $p^{-1}=1/2+q^{-1}$, i.e. $q=2p/(2-p)$, H\"older's
  inequality yields
  \[
  \norm[\big]{[M,M]_\infty^{1/2}}_{\L_p} =
  \norm[\big]{[M,M]_\infty^{1/2}H_\infty H_\infty^{-1}}_{\L_p} \leq
  \norm[\big]{[M,M]_\infty^{1/2}H_\infty}_{\L_2} \norm[\big]{H_\infty^{-1}}_{\L_q},
  \]
  where, by definition of $H$ and elementary computations,
  \[
  \norm[\big]{H_\infty^{-1}}_{\L_q} =
  \norm[\big]{\varepsilon+M_\infty^*+D_\infty}_{\L_p}^{1-p/2}.
  \]
  Moreover, since $s \mapsto H_s$ is decreasing, one has $[N,N]_\infty
  = \bigl( H_-^2 \cdot [M,M] \bigr)_\infty \geq H^2_\infty
  [M,M]_\infty$, which implies
  $\norm[\big]{[M,M]_\infty^{1/2}H_\infty}_{\L_2} \leq
  \norm[\big]{[N,N]_\infty^{1/2}}_{\L_2}$, hence also
  \[
  \norm[\big]{[M,M]_\infty^{1/2}}_{\L_p} \leq
  \norm[\big]{\varepsilon+M_\infty^*+D_\infty}_{\L_p}^{1-p/2}
  \norm[\big]{[N,N]_\infty^{1/2}}_{\L_2} = 
  \norm[\big]{\varepsilon+M_\infty^*+D_\infty}_{\L_p}^{1-p/2}
  \norm[\big]{N_\infty}_{\L_2(H)}. 
  \]
  The integration-by-parts formula
  \[
  N_\infty = \bigl( H_- \cdot M \bigr)_\infty 
  = H_\infty M_\infty + \int_0^\infty M_s\,d(-H_s) 
  \]
  and the inequality $\|M_s\| \leq \|M_{s-}\| + \|\Delta M_s\| \leq
  \varepsilon + M^*_{s-} + D_{s-}$ (valid for all $s\geq 0$) yield
  \begin{align*}
    N^*_\infty &= (\varepsilon + D_\infty + M^*_\infty)^{p/2-1} M_\infty
    + \int_0^\infty M_s\,d\bigl(-(\varepsilon + D_s + M^*_s)^{p/2-1}\bigr)\\
    &\leq (\varepsilon + D_\infty + M^*_\infty)^{p/2} + \int_0^\infty
    (\varepsilon + M^*_{s-} + D_{s-}) \,d\bigl(-(\varepsilon + D_s
    + M^*_s)^{p/2-1}\bigr).
  \end{align*}
  Setting $V:=\varepsilon + M^* + D$ and appealing to Lemma
  \ref{lm:fv1}, one obtains the estimate
  \begin{align*}
    &\int_0^\infty (\varepsilon + M^*_{s-} + D_{s-})
    \,d\bigl(-(\varepsilon + D_s + M^*_s)^{p/2-1}\bigr)\\
    &\hspace{4em}= \int_0^\infty V_- \,d(-V^{p/2-1})
    \leq  \frac{2-p}{p} V_\infty^{p/2}
    = \frac{2-p}{p} \bigl( \varepsilon + M^*_\infty + D_\infty \bigr)^{p/2},
  \end{align*}
  which yields
  \[
  \norm[\big]{N_\infty}_{\L_2(H)} \leq \norm[\big]{N^*_\infty}_{\L_2} \leq
  \frac2p \norm[\big]{\bigl( \varepsilon + M^*_\infty + D_\infty
    \bigr)^{p/2}}_{\L_2} 
  = \frac2p \norm[\big]{\varepsilon + M^*_\infty + D_\infty}_{\L_p}^{p/2}.
  \]
  Collecting estimates, one concludes that
  \[
  \norm[\big]{[M,M]_\infty^{1/2}}_{\L_p} \leq \frac2p
  \norm[\big]{\varepsilon + M^*_\infty + D_\infty}_{\L_p}.
  \]
  The proof is completed by recalling that $\varepsilon>0$ is
  arbitrary, hence the last inequality holds also with
  $\varepsilon=0$.
\end{proof}

\begin{prop}[Conditional upper bound, $1 < p < 2$]     \label{prop:cub2-}
  Let $M$ be an $\H$-valued martingale and $D$ an increasing adapted
  process such that $\|\Delta M_-\| \leq D_-$. Then one has, for any
  $p \in ]1,2[$,
  \[
  \norm[\big]{M_\infty^*}_{\L_p} \lesssim_p
  \norm[\big]{[M,M]_\infty^{1/2} + D_\infty}_{\L_p}.
  \]
\end{prop}
\begin{proof}
  Let $\varepsilon>0$ be arbitrary, and define the processes
  \[
  H := \sqrt{p/2} \, \bigl( \varepsilon + [M,M] + D^2 \bigr)^{p/4-1/2},
  \qquad N:= H_- \cdot M,
  \]
  so that $M= H_-^{-1} \cdot N$, hence also, exactly as in the proof
  of Proposition \ref{prop:cub1}, $M_\infty^* \leq 2 H_\infty^{-1}
  N_\infty^*$, which in turn implies, by H\"older's inequality,
  \[
  \norm[\big]{M_\infty^*}_{\L_p} \leq 2 \norm[\big]{H_\infty^{-1} N_\infty^*}_{\L_p}
  \leq 2 \norm[\big]{H_\infty^{-1}}_{\L_q} \norm[\big]{N_\infty^*}_{\L_2},
  \]
  where $q>0$ is defined by $p^{-1}=1/2+q^{-1}$, i.e. $q=2p/(2-p)$. By
  definition of $H$, one has
  \begin{align*}
    \norm[\big]{H_\infty^{-1}}_{\L_q} &= \sqrt{2/p} \,
    \norm[\big]{\bigl( \varepsilon + [M,M]_\infty + D_\infty^2 
           \bigr)^{1/2(1-p/2)}}_{\L_q}\\
    &\leq \sqrt{2/p} \, \norm[\big]{\bigl( \varepsilon + [M,M]^{1/2} + D 
                               \bigr)^{1-p/2}}_{\L_q}
    = \sqrt{2/p} \, \norm[\big]{\varepsilon + [M,M]^{1/2} + D}_{\L_p}^{1-p/2}.
  \end{align*}
  Moreover, since
  \[ 
  [N,N]_\infty = \bigl( H_-^2 \cdot [M,M] \bigr)_\infty = 
  \frac{p}{2} \int_0^\infty \bigl( \varepsilon + [M,M]_{s-} + D_{s-}^2 
  \bigr)^{p/2-1}\,d[M,M]_s
  \]
  and $[M,M]_{s-} + D_{s-}^2 \geq [M,M]_{s-} + \|\Delta M_s\|^2 =
  [M,M]_s$, one has
  \[
  [N,N]_\infty \leq \frac{p}{2} \int_0^\infty
  \bigl( \varepsilon + [M,M]_s\bigr)^{p/2-1}\,d(\varepsilon + [M,M]_s),
  \]
  because $x \mapsto x^{p/2-1}$ is decreasing (recall that $p/2-1<0$).
  Setting $V=\varepsilon + [M,M]$ and recalling the
  integration-by-parts formula \eqref{eq:ibp-fv} as well as Lemma
  \ref{lm:fv1}, one has
  \begin{multline*}
  \int_0^\infty \bigl( \varepsilon + [M,M]_s
  \bigr)^{p/2-1}\,d(\varepsilon + [M,M]_s) \\
  = V_\infty^{p/2} + 
  \int_0^\infty V_{s-} \,d\bigl(-V_s^{p/2-1}\bigr)
  \leq \frac2p V^{p/2} = \frac2p \bigl( \varepsilon + [M,M]_\infty \bigr)^{p/2},
  \end{multline*}
  in particular $[N,N]^{1/2}_\infty \leq (\varepsilon +
  [M,M]^{1/2}_\infty)^{p/2}$.  Doob's inequality then yields
  \begin{align*}
  \norm[\big]{N^*_\infty}_{\L_2} &\leq 2\norm[\big]{N_\infty}_{\L_2(H)}
  = 2 \norm[\big]{[N,N]^{1/2}_\infty}_{\L_2}\\
  &\leq 2 \norm[\big]{\bigl(
   \varepsilon + [M,M]^{1/2}_\infty \bigr)^{p/2}}_{\L_2}
  \leq 2 \norm[\big]{\varepsilon + [M,M]^{1/2}_\infty + D_\infty}^{p/2}_{\L_p}.
  \end{align*}
  Collecting estimates, one ends up with
  \[
  \norm[\big]{M^*_\infty}_{\L_p} \leq 4\sqrt{2/p} \,
  \norm[\big]{\varepsilon + [M,M]^{1/2}_\infty + D_\infty}_{\L_p}.
  \]
  The proof is completed, once again, simply by recalling that
  $\varepsilon>0$ is arbitrary.
\end{proof}

\begin{prop}[Conditional lower bound, $p>2$]     \label{prop:clb2+}
  Let $M$ be an $\H$-valued martingale and $D$ an increasing adapted
  process such that $\|\Delta M_-\| \leq D_-$. Then one has, for any
  $p \in \, ]{2,\infty}[$,
  \[
  \norm[\big]{[M,M]_\infty^{1/2}}_{\L_p} \lesssim_p 
  \norm[\big]{M^*_\infty}_{\L_p} + \norm[\big]{D_\infty}_{\L_p}
  \]
\end{prop}
\begin{proof}
  Let us set
  \[
  H=\sqrt{p/2} \, \bigl( [M,M]+D^2 \bigr)^{p/4-1/2}, 
  \qquad N=H_- \cdot M,
  \]
  so that
  \[ 
  [N,N]_\infty = \frac{p}{2} \int_0^\infty \bigl( 
  [M,M]_{s-} + D_{s-}^2 \bigr)^{p/2-1}\,d[M,M]_s
  \geq \frac{p}{2} \int_0^\infty [M,M]_s^{p/2-1}\,d[M,M]_s,
  \]
  where we have used the inequality $[M,M]_- + D_-^2 \geq [M,M]_- +
  \|\Delta M\|^2 = [M,M]$. Let us show that $[N,N]_\infty \geq
  [M,M]_\infty^{p/2}$: setting $V:=[M,M]$ and $q:=p/2$ for
  notational convenience, one has, using the integration-by-parts
  formula \eqref{eq:ibp-fv},
  \[ 
  [N,N]_\infty \geq q \int_0^\infty V^{q-1}\,dV
  = q \left( V_\infty^q - \int_0^\infty V_-\,d(V^{q-1}) \right).
  \]
  Lemma \ref{lm:fv1} then yields
  \[ 
  [N,N]_\infty \geq q \left( V_\infty^q - \int_0^\infty V_-\,d(V^{q-1}) \right)
  \geq q(1-(q-1)/q) V_\infty^q = [M,M]^{p/2}_\infty,
  \]
  hence, immediately,
  \[
  \norm[\big]{[M,M]^{1/2}_\infty}_{\L_p} \leq
  \norm[\big]{[N,N]^{1/2}_\infty}^{2/p}_{\L_2} =
  \norm[\big]{N_\infty}^{2/p}_{\L_2(H)} \leq \norm[\big]{N^*_\infty}^{2/p}_{\L_2}.
  \]
  Moreover, integrating by parts, one has
  \[
  N_\infty = (H_- \cdot M)_\infty = H_\infty M_\infty - \int_0^\infty M_t\,dH_t,
  \]
  which implies $N_\infty^* \leq 2 H_\infty M^*_\infty$, thus also,
  setting $q:=2p/(p-2)$,
  \[
  \norm[\big]{N^*_\infty}_{\L_2} \leq 2 \norm[\big]{H_\infty M^*_\infty}_{\L_2}
  \leq 2 \norm[\big]{H_\infty}_{\L_q} \norm[\big]{M^*_\infty}_{\L_p}
  \leq \sqrt{2p} \, 
  \norm[\big]{\bigl([M,M]_\infty+D^2_\infty\bigr)^{1/2}}^{p/2-1}_{\L_p} 
  \norm[\big]{M^*_\infty}_{\L_p}.
  \]
  The last two estimates and the elementary inequality $(a+b)^{1/2}
  \leq a^{1/2}+b^{1/2}$ yield
  \[
  \norm[\big]{[M,M]^{1/2}_\infty}_{\L_p} \leq (2p)^{1/p} \,
  \norm[\big]{[M,M]^{1/2}_\infty+D_\infty}^{1-2/p}_{\L_p}
  \norm[\big]{M^*_\infty}^{2/p}_{\L_p}.
  \]
  Let us apply Young's inequality in the form
  \[
  ab \leq \varepsilon \frac{a^q}{q} + N(\varepsilon) \frac{b^{q'}}{q'},
  \qquad \frac1q + \frac1{q'}=1, \qquad
  N(\varepsilon) = \varepsilon^{-1/(q-1)},
  \]
  choosing $q=p/(p-2)$, hence $q'=p/2$, and $\varepsilon=\ds \frac{q}2
  (2p)^{-1/p}$: we obtain
  \begin{align*}
  \norm[\big]{[M,M]^{1/2}_\infty}_{\L_p} &\leq
  \frac12 \norm[\big]{[M,M]^{1/2}_\infty+D_\infty}_{\L_p} 
  + N(p) \norm[\big]{M^*_\infty}_{\L_p}\\
  &\leq
  \frac12 \norm[\big]{[M,M]^{1/2}_\infty}_{\L_p} 
  + \frac12 \norm[\big]{D_\infty}_{\L_p} 
  + N(p) \norm[\big]{M^*_\infty}_{\L_p},
  \end{align*}
  with $N(p) = 2^{p/2+1}$ (which is not the optimal constant), hence,
  collecting terms,
  \[
  \norm[\big]{[M,M]^{1/2}_\infty}_{\L_p} \leq 2^{p/2+2} \,
  \norm[\big]{M^*_\infty}_{\L_p} + \norm[\big]{D_\infty}_{\L_p}.  \qedhere
  \]
\end{proof}

\subsection{Davis' decomposition}
The following result, known as Davis' decomposition (whose
continuous-time adaptation is due to Meyer \cite{Mey:dual}), is the
key to extend the maximal estimates of the previous subsection to
general martingales without any assumption on their jumps. The proof
we give here follows closely \cite{Mey:dual}.
\begin{lemma}[Davis' decomposition]     \label{lm:davis}
  Let $M$ be an $\H$-valued martingale, and define $S:=(\Delta M)^*$,
  i.e. $S_t := \sup_{s \leq t} \norm{\Delta M_s}$ for all $t \geq
  0$. Then there exist martingales $L$ and $K$ such that
  \begin{itemize}
  \item[\emph{(i)}] $M=L+K$;
  \item[\emph{(ii)}] $\norm{\Delta L} \leq 4S_-$;
  \item[\emph{(iii)}] $K=K^1-K^2$, where $\int_0^\infty |dK^1| \leq
    2S_\infty$ and $K^2$ is the compensator of $K^1$.
  \end{itemize}
\end{lemma}
\begin{proof}
  Let us define the process $K^1$ by
  \begin{equation}     \label{eq:zompa}
  K_t^1 := \sum_{s \leq t} \Delta M_s 1_{\{\| \Delta M_s \| \geq
    2S_{t-}\}} \qquad \forall t \geq 0.
  \end{equation}
  Note that, if $t$ is such that $\| \Delta M_t \| \geq 2S_{t-}$, one
  has
  \[
  \| \Delta M_t \| + 2S_{t-} \leq 2 \| \Delta M_t \| = 2S_t,
  \]
  hence $\| \Delta M_t \| \leq 2(S_t-S_{t-})$, which implies that the
  process $K^1$ has variation bounded by $2S_\infty$. Let $K=K^1-K^2$,
  where $K^2$ is the (predictable) compensator of $K^1$, so that, in
  particular, $K$ is a martingale.  Furthermore, let us set
  $L^1:=M-K^1$, $L^2:=-K^2$, and $L=L^1-L^2$. In particular, $L=M-K$
  is a martingale. It remains only to prove (ii): since $L=L^1+K^2$
  and $K^2$ is predictable, it follows that, for any totally
  inaccesible jump time $T$, one has $\Delta L_T=\Delta L^1_T$
  (e.g. by \cite[Prop.~7.7]{Met}), hence also, by \eqref{eq:zompa} and
  by definition of $L^1$, $\norm{\Delta L_T} \leq 2S_{T-}$.  Moreover,
  for any preditable jump time $T$, one has $\E[\Delta
  L_T|\mathcal{F}_{T-}]=0$ because $L$ is a martingale, and $\E[\Delta
  L^2_T|\mathcal{F}_{T-}]=\Delta L^2_T$ because $L^2$ is predictable,
  thus also $\Delta L^2_T = -\E[\Delta L^1_T|\mathcal{F}_{T-}]$. This
  implies, again by \eqref{eq:zompa} and by definition of $L^1$, that
  $\norm{\Delta L^2_T} \leq 2 \E[S_{T-}|\mathcal{F}_{T-}] = 2S_{T-}$,
  hence also that $\norm{\Delta L_T} \leq \norm{\Delta L^1_T} +
  \norm{\Delta L^2_T} \leq 4S_{T-}$.  By general decomposition results
  for stopping times (see e.g. \cite[\S7]{Met}), it follows that
  $\norm{\Delta L} \leq 4S_-$.
\end{proof}

\subsection{Real martingales}
Throughout this section we shall use the symbols $L$, $K$ and $S$ as
they have been defined in Lemma \ref{lm:davis} above.

We start with a simple corollary of Davis' decomposition and
Proposition \ref{prop:BDG-fv}.
\begin{lemma}     \label{lm:dKS}
  Let $M$ be a real martingale with Davis' decomposition $M=K+L$. Then
  one has $\norm*{\int_0^\infty |dK|}_{\L_p} \lesssim_p
  \norm[\big]{S_\infty}_{\L_p}$.
\end{lemma}
\begin{proof}
  Note that we can write $K^1=K^{1+}+K^{1-}$, where $K^{1+}$ has only
  positive jumps and $K^{1-}$ has only negative jumps. Then it is
  immediate to see that $K^2=\widetilde{K^{1+}}+\widetilde{K^{1-}}$,
  which, taking Proposition \ref{prop:BDG-fv} into account, implies
  that
  \begin{align*}
  \norm*{\int_0^\infty |dK|}_{\L_p} &\leq
  \norm*{\int_0^\infty |dK^{1+}|}_{\L_p} +
  \left\| \int_0^\infty |d\widetilde{K^{1+}}| \right\|_{\L_p} +
  \left\| \int_0^\infty |dK^{1-}| \right\|_{\L_p} +
  \left\| \int_0^\infty |d\widetilde{K^{1-}}| \right\|_{\L_p}\\
  & \leq (p+1) \left\| \int_0^\infty |dK^{1+}| \right\|_{\L_p}
  + (p+1) \left\| \int_0^\infty |dK^{1-}| \right\|_{\L_p}.
  \end{align*}
  Since the total variation of both $K^{1+}$ and $K^{1-}$ is bounded
  by $2S_\infty$, we conclude that
  \[
  \norm*{\int_0^\infty |dK|}_{\L_p} \leq 4(p+1)
  \norm[\big]{S_\infty}_{\L_p}.
  \qedhere
  \]
\end{proof}

\begin{proof}[Proof of Theorem \ref{thm:BDG} ($\H=\erre$)]
  Note that the upper bound for $p\geq 2$ has already been addressed. We
  treat the remaining three cases separately.

  \noindent\textsc{Lower bound, $1 \leq p<2$.} Let $M=L+K$ be the Davis'
  decomposition of the martingale $M$. Then one has
  $[M,M]^{1/2}_\infty \leq [L,L]^{1/2}_\infty + [K,K]^{1/2}_\infty$,
  hence also
  \[
  \norm[\big]{[M,M]^{1/2}_\infty}_{\L_p} \leq
  \norm[\big]{[L,L]^{1/2}_\infty}_{\L_p} +
  \norm[\big]{[K,K]^{1/2}_\infty}_{\L_p},
  \]
  where, by Propositions \ref{prop:clb1} and \ref{prop:clb2-},
  \[
  \norm[\big]{[L,L]^{1/2}_\infty}_{\L_p} \lesssim_p
  \norm[\big]{L^*_\infty}_{\L_p} + \norm[\big]{S_\infty}_{\L_p}.
  \]
  Note that $M=L+K$ also implies $L^* \leq M^*+K^*$, and, by
  definition of $S$, it is immediate that $S_\infty \leq 2
  M^*_\infty$. Therefore, taking Lemma \ref{lm:dKS} into
  account, one has
  \begin{align*}
  \norm[\big]{[M,M]^{1/2}_\infty}_{\L_p} &\lesssim_p
  \norm[\big]{M^*_\infty}_{\L_p} + \norm[\big]{K^*_\infty}_{\L_p}
  + \norm[\big]{[K,K]^{1/2}_\infty}_{\L_p} \leq
  \norm[\big]{M^*_\infty}_{\L_p} + 2 \left\| \int_0^\infty |dK| \right\|_{\L_p}\\
  &\lesssim_p \norm[\big]{S_\infty}_{\L_p} \lesssim \norm[\big]{M^*_\infty}_{\L_p}.
  \qedhere
  \end{align*}

\medskip

  \noindent\textsc{Upper bound, $1 \leq p < 2$.}
  Taking into account Propositions \ref{prop:cub1} and
  \ref{prop:cub2-}, the estimates $\|M_\infty^*\|_{\L_p} \leq
  \|L_\infty^*\|_{\L_p} + \|K_\infty^*\|_{\L_p}$ and $[L,L]^{1/2}
  \leq [M,M]^{1/2} + [K,K]^{1/2}$ yield
  \begin{align*}
  \bigl\| M_\infty^* \bigr\|_{\L_p} &\lesssim_p \bigl\|
  [L,L]_\infty^{1/2} \bigr\|_{\L_p} + \bigl\| S_\infty \bigr\|_{\L_p}
  + \bigl\| K_\infty^* \bigr\|_{\L_p}\\
  &\leq \norm[\big]{[M,M]_\infty^{1/2}}_{\L_p} + \norm[\big]{S_\infty}_{\L_p}
  + \norm[\big]{[K,K]_\infty^{1/2}}_{\L_p} + \norm[\big]{K^*_\infty}_{\L_p}.
  \end{align*}
  Noting that 
  \[
  S_\infty = (\Delta M)^* = \bigl( \bigl(|\Delta M|^2\bigr)^* \bigr)^{1/2}
  \leq \Bigl( \sum |\Delta M|^2 \Bigr)^{1/2} \leq [M,M]_\infty^{1/2},
  \]
  hence $\norm[\big]{S_\infty}_{\L_p} \leq
  \norm[\big]{[M,M]_\infty^{1/2}}_{\L_p}$, we are left with
  \[
  \norm[\big]{M_\infty^*}_{\L_p} \lesssim_p
  \norm[\big]{[M,M]_\infty^{1/2}}_{\L_p}
  + \norm[\big]{[K,K]_\infty^{1/2}}_{\L_p}
  + \norm[\big]{K^*_\infty}_{\L_p},
  \]
  where, repeating an argument used in the previous part of the proof,
  \[
  \norm[\big]{[K,K]_\infty^{1/2}}_{\L_p}
  + \norm[\big]{K^*_\infty}_{\L_p}
  \lesssim_p
  \left\| \int_0^\infty |dK|\, \right\|_{\L_p}
  \lesssim \bigl\| S_\infty \bigr\|_{\L_p} \leq 
  \bigl\| [M,M]_\infty^{1/2} \bigr\|_{\L_p}.
  \]
  Collecting estimates, the proof is completed.

\medskip

  \noindent\textsc{Lower bound, $p>2$.} Davis' decomposition $M=L+K$ implies
  $[M,M]_\infty^{1/2} \leq [L,L]_\infty^{1/2} + [K,K]_\infty^{1/2}$,
  hence also, taking Proposition \ref{prop:clb2+} into account,
  \[
  \norm[\big]{[M,M]^{1/2}_\infty}_{\L_p} \leq 
  \norm[\big]{[L,L]^{1/2}_\infty}_{\L_p} +
  \norm[\big]{[K,K]^{1/2}_\infty}_{\L_p}
  \lesssim_p \norm[\big]{L^*_\infty}_{\L_p} + \norm[\big]{S_\infty}_{\L_p}
  + \norm[\big]{[K,K]^{1/2}_\infty}_{\L_p}.
  \]
  We can further estimate the terms on the right-hand side as follows:
  \begin{align*}
    \norm[\big]{L^*_\infty}_{\L_p} &\leq \norm[\big]{M^*_\infty}_{\L_p} +
    \norm[\big]{K^*_\infty}_{\L_p} \leq \norm[\big]{M^*_\infty}_{\L_p} +
    \left\| \int_0^\infty |dK| \right\|_{\L_p},\\
    \norm[\big]{S_\infty}_{\L_p} &\leq 2 \norm[\big]{M^*_\infty}_{\L_p},\\
    \norm[\big]{[K,K]^{1/2}_\infty}_{\L_p} &\leq \left\| \int_0^\infty |dK|
    \right\|_{\L_p},
  \end{align*}
  hence the proof is complete, recalling that by Lemma \ref{lm:dKS}
  \[
  \left\| \int_0^\infty |dK| \right\|_{\L_p} \lesssim_p 
  \norm[\big]{S_\infty}_{\L_p} \lesssim \norm[\big]{M^*_\infty}_{\L_p}.
  \qedhere
  \]
\end{proof}

\subsection{Hilbert-space-valued martingales}
\begin{lemma}[Upper bounds by lower bounds] \label{lm:ubyl} Let $p \in
  ]{1,\infty}[$ and assume that
  \[
  \norm[\big]{[N,N]_\infty^{1/2}}_{\L_p} \lesssim_p
  \norm[\big]{N^*_\infty}_{\L_p}
  \]
  for any $\H$-valued martingale $N$. Then, for any $\H$-valued
  martingale $M$, one has
  \[
  \norm[\big]{M^*_\infty}_{\L_{p'}} \lesssim_{p'}
  \norm[\big]{[M,M]_\infty^{1/2}}_{\L_{p'}},
  \]
  where $p'$ is the conjugate exponent of $p$, i.e. $1/p+1/p' = 1$.
\end{lemma}
\begin{proof}
  Note that the (topological and algebraic) dual of $\L_p(\H)$ is
  $\L_{p'}(\H)$, with duality form $\E(\cdot,\cdot)$. Therefore,
  denoting the unit ball of $\L_p(\H)$ by $B_1$, one has, recalling
  Doob's inequality,
  \[
  \norm[\big]{M^*_\infty}_{\L_{p'}} \lesssim_{p'} 
  \norm[\big]{M_\infty}_{\L_{p'}(\H)} = \sup_{\xi \in B_1} \E(M_\infty,\xi).
  \]
  Let $\xi \in \L_p(\H)$, $\E\norm{\xi}^p\leq 1$, be arbitrary, and
  consider the martingale $t \mapsto N_t:=\E[\xi|\mathcal{F}_t]$,
  $N_\infty:=\xi$. Kunita-Watanabe's and H\"older's inequalities yield
  \begin{align*}
  \E(M_\infty,\xi) \equiv \E(M_\infty,N_\infty)
  &= \E[M,N]_\infty\\
  &\leq \E[M,M]_\infty^{1/2}[N,N]_\infty^{1/2}
  \leq \norm{[M,M]_\infty^{1/2}}_{\L_{p'}} \norm{[N,N]_\infty^{1/2}}_{\L_p}.
  \end{align*}
  By the hypothesis and Doob's inequality, one also has
  \[
  \norm{[N,N]_\infty^{1/2}}_{\L_p} \lesssim_p
  \norm{N_\infty}_{\L_p(\H)} \equiv \norm{\xi}_{\L_p(\H)} \leq 1,
  \]
  therefore we are left with $\E(M_\infty,\xi) \lesssim_p
  \norm{[M,M]_\infty^{1/2}}_{\L_{p'}}$. Since $\xi\in B_1$ is
  arbitrary, the claim is proved.
\end{proof}

\begin{lemma}[Interpolation of lower bounds]     \label{lm:interp}
  Let $1 < p_1 < p_2 < \infty$, and assume that, for any $\H$-valued
  martingale $M$, one has
  \[
  \norm[\big]{[M,M]_\infty^{1/2}}_{\L_{p_i}} \lesssim_{p_i}
  \norm[\big]{M^*_\infty}_{\L_{p_i}} \qquad \forall i \in \{1,2\}.
  \]
  Then one has $\norm[\big]{[M,M]_\infty^{1/2}}_{\L_p} \lesssim_p
  \norm[\big]{M^*_\infty}_{\L_p}$ for all $p \in [p_1,p_2]$.
\end{lemma}
\begin{proof}
  Thanks to the martingale convergence theorem and Doob's inequality,
  we can identify the space of $\H$-valued martingales such that
  $\norm{M^*_\infty}_{\L_p}<\infty$, $p>1$, with
  $L_p(\Omega \to \H,\mathcal{F}_\infty,\P)$. For $i\in\{1,2\}$, let us
  define, with a slight abuse of notation, the operator
  \begin{align*}
    T: \L_{p_i}(\H) &\to \L_{p_i}\\
       M_\infty &\mapsto [M,M]^{1/2}_\infty,
  \end{align*}
  which is immediately seen to be sublinear in the sense of
  \S\ref{ssec:RT}. Taking Doob's inequality into account, one has
  \[
    \norm[\big]{TM}_{\L_{p_i}} \equiv \norm[\big]{[M,M]_\infty^{1/2}}_{\L_{p_i}}
    \lesssim_{p_i} \norm[\big]{M_\infty^*}_{\L_{p_i}} 
    \lesssim_{p_i} \norm[\big]{M_\infty}_{\L_{p_i}(\H)} 
    \qquad \forall i \in \{1,2\},
  \]
  i.e. $T$ is bounded from $\L_{p_1}(\H)$ to $\L_{p_1}$, and
  from $\L_{p_2}(\H)$ to $\L_{p_2}$. Therefore, by Theorem \ref{thm:CZ} we
  infer that $T$ is bounded from $\L_p(\H)$ to $\L_p$ for all $p \in
  \mathopen]p_1,p_2\mathclose[$.
\end{proof}

\begin{proof}[Proof of Theorem \ref{thm:BDG}]
  If $p=1$, the proof of the previous section still works, if one
  appeals to Proposition \ref{prop:veme}, which implies that
  \[
  \norm*{\int_0^\infty |dK|}_{\L_1} \lesssim 
  \norm*{\int_0^\infty |dK^1|}_{\L_1} \lesssim
  \norm{S_\infty}_{\L_1} \lesssim \norm[\big]{M^*_\infty}_{\L_1}.
  \]
  By the results in Section \ref{sec:1st} we know that the upper bound
  holds for all $p \geq 2$, and that the lower bound holds for $p=2$
  and for all $p \geq 4$. Therefore, by Lemma \ref{lm:interp}, the
  lower bound in fact holds for all $p \geq 2$. In turn, Lemma
  \ref{lm:ubyl} yields the validity of the upper bound for all $p \in
  ]1,2[$ (hence for all $p \geq 1$), provided the lower bound holds
  for all $p \in ]1,2[$, which we shall prove now: letting $M=L+K$ be
  the Davis decomposition of $M$ and proceeding as in the proof for
  real-valued martingales of the previous section, one obtains
  \[
  \norm[\big]{[M,M]_\infty^{1/2}}_{\L_p} \lesssim_p
  \norm[\big]{M^*_\infty}_{\L_p} +
  \norm[\big]{K^*_\infty}_{\L_p} +
  \norm[\big]{[K,K]_\infty^{1/2}}_{\L_p}.
  \]
  Applying the upper bound to the martingale $K$, one has
  $\norm[\big]{K^*_\infty}_{\L_p} \lesssim_p
  \norm[\big]{[K,K]_\infty^{1/2}}_{\L_p}$, thus also
  \[
  \norm[\big]{[M,M]_\infty^{1/2}}_{\L_p} \lesssim_p
  \norm[\big]{M^*_\infty}_{\L_p} +
  \norm[\big]{[K,K]_\infty^{1/2}}_{\L_p}.
  \]
  To complete the proof, we only have to show that the second term on
  the right-hand side is controlled by the $\L_p$-norm of
  $M^*_\infty$. To this purpose, recall that $K=K^1-K^2$, where $K^1$,
  $K^2$ have integrable variation and $K^2$ is the (predictable)
  compensator of $K^1$. Therefore one has
  \[ 
  [K,K]^{1/2}_\infty \leq [K^1,K^1]^{1/2}_\infty + [K^2,K^2]^{1/2}_\infty
  \]
  and, by a reasoning already used in the proof of Davis'
  decomposition,
  \[
  [K^2,K^2]_\infty = \sum_{n\in\enne} \norm{\Delta K^2_{T_n}}^2,
  \]
  where $\Delta K^2_{T_n} = \E[\Delta K_{T_n}^1|\mathcal{F}_{T_n-}]$
  and $(T_n)_{n\in\enne}$ is a sequence of predictable stopping times.
  Applying Theorem \ref{thm:stein} to the discrete-time process $n
  \mapsto K^1_{T_n}$, one obtains
  $\norm[\big]{[K^2,K^2]^{1/2}_\infty}_{\L_p} \lesssim_p
  \norm[\big]{[K^1,K^1]^{1/2}_\infty}_{\L_p}$, hence also
  \begin{align*}
  \norm[\big]{[K,K]^{1/2}_\infty}_{\L_p} &\lesssim
  \norm[\big]{[K^1,K^1]^{1/2}_\infty}_{\L_p} +
  \norm[\big]{[K^2,K^2]^{1/2}_\infty}_{\L_p}\\
  &\lesssim_p \norm[\big]{[K^1,K^1]^{1/2}_\infty}_{\L_p}
  \leq \norm*{\int_0^\infty|dK^1|}_{\L_p}
  \lesssim \norm[\big]{M^*_\infty}_{\L_p}.
  \qedhere
  \end{align*}
\end{proof}

%\newpage
\let\oldbibliography\thebibliography
\renewcommand{\thebibliography}[1]{%
  \oldbibliography{#1}%
  \setlength{\itemsep}{-1pt}%
}

\bibliographystyle{amsplain}
\bibliography{ref}

\end{document}